\documentclass[11pt,reqno]{amsart}
\usepackage{mathrsfs}
\usepackage{url}
\usepackage{mathtools}
\usepackage{latexsym,epsfig,amssymb,amsmath,amsthm,color,url,bm}
\usepackage[inline,shortlabels]{enumitem}
\usepackage{hyperref}
\usepackage[foot]{amsaddr}
\usepackage{amsmath,amsbsy}
\usepackage{mwe}
\RequirePackage[numbers]{natbib}
\usepackage{mathptmx}
\usepackage[text={16cm,24cm}]{geometry}

\allowdisplaybreaks 
 \numberwithin{equation}{section}
\newtheorem{theorem}{Theorem}[]

\newtheorem{lemma}[theorem]{Lemma}

\newcommand\Item[1][]{%
  \ifx\relax#1\relax  \item \else \item[#1] \fi
  \abovedisplayskip=0pt\abovedisplayshortskip=0pt~\vspace*{-\baselineskip}}

\theoremstyle{definition}

\theoremstyle{definition}

\theoremstyle{definition}
\newtheorem{corollory}[theorem]{Corollory}

\theoremstyle{definition}

\title{The direct product of a star and a path is antimagic}
\date{}
\author{Vinothkumar Latchoumanane$^{1}$, Murugan Varadhan$^{2}$, Andrea Semani\v{c}ov\'{a}-Fe\v{n}ov\v{c}\'{i}kov\'{a}$^{3,4}$}
\address{${}^{1,2}$Department \ of \ Mathematics, \ School \ of \ Advanced \ Sciences, \  Vellore \ Institute \ o f\  Technology, \  Vellore \ - \ $632014$,\ India.} 
\address{${}^{3}$ Department \ of  \ Applied  \ Mathematics \ and \ Informatics, \ Technical \ University, \ Letn\'a~9, \ Ko\v{s}ice,  \ Slovak \ Republic}
\address{${}^{4}$ Division \ of \ Mathematics, \ Saveetha \ School \  of \ Engineering, \ SIMATS, \ Chennai, \ India}

\email{${}^{1}$vinothkumar.l2019@vitstudent.ac.in, ${}^{2}$murugan.v@vit.ac.in, ${}^{3,4}$andrea.fenovcikova@tuke.sk}

\begin{document}
\bibliographystyle{plainnat}

\begin{abstract}
A graph $G$ is antimagic if there exists a bijection $f$ from $E(G)$ to $\left\{1,2, \dots,|E(G)|\right\}$ such that the vertex sums for all vertices of $G$ are distinct, where the vertex sum  is defined as the sum of the labels of all  incident edges.
 Hartsfield and Ringel   conjectured that every connected graph other than $K_2$ admits an antimagic labeling. It is still a challenging problem to address antimagicness in the case of disconnected graphs. In this paper, we study antimagicness for the disconnected graph that is constructed as the direct product of a star and a path.
\end{abstract}

\subjclass[2020]{}Primary 05C78, Secondary 05C76


\maketitle

\section{Introduction}
In 1990, the concept of antimagic labeling was first introduced by Hartsfield and Ringel \cite{Ringel1}. They called a graph $G$ {\it antimagic} if there exists a bijection $f: E(G) \to \left\{1,2, \dots, \left|E(G)\right|\right\}$ such that for all vertices their weights,   defined as the sum of all the incident edge labels,  are distinct.  In the same introduction paper \cite{Ringel1} some simple graphs such as paths, cycles, complete graphs and wheels   are proven to be antimagic and, Hartsfield and Ringel posed the strong conjecture that every connected graph except $K_2$ is antimagic.

 Although several researchers have attempted to settle the above conjecture,  still the conjecture remains open. For instance, Alon et al., \cite{Alon} validated the conjecture for graphs with sufficiently large minimum degree. They proved  that there exists an absolute constant $c$ such that every graph with $n$ vertices and minimum degree at least $c \log n$ is antimagic. They also proved that every graph with at least for vertices and the maximum degree $\Delta (G) \geq n-2$ is antimagic. A similar result was improved by Yilma \cite{Yilma} for $\Delta (G) \geq n-3$.
  
Only some authors studied antimagicness for disconnected graphs. Not all   disconnected graphs admit an antimagic labeling. Trivially, the graph $K_2 \cup G$ is not antimagic for any graph $G$. Exploring the set of all disconnected antimagic graphs seems to be another interesting open problem.
 In this context, Wang et al. \cite{WANG} studied antimagicness of  unions cycles, stars and paths. More precisely, they proved that $C_n \cup K_{1,k}$ for $n \geq 3$ and $k \geq 2 \sqrt{n} + 2$, two copies of a path $2P_n$, $n\ge4$, and a union of a cycle and a path, $K_{1,n} \cup P_n$ and $K_{1,n} \cup P_{n+1}$  for $n \geq 3$ are antimagic graphs.
  Shang et al. \cite{ShangStar} considered a star forest with at most one star $K_{1,2}$ as its components and no component  isomorphic to $K_{1,1}$. They proved that $mK_{1,2} \cup K_{1,n}$, $n\geq 3$ is antimagic if and only if $m \leq \min \left\{2n+1,(2n-5+\sqrt{8n^2-24n+17})/2\right\}$. Shang    \cite{Shang2}  showed that a linear forest with no component isomorphic to $P_2$, $P_3$ and $P_4$ is antimagic. 
  Chen et al. \cite{Chenn}   proved that $mK_{1,2} \cup K_{1,n}$, $ n\geq 3$ is antimagic if and only if $n \geq \max \left\{(m-1)/{2},(1-2m+\sqrt{8m^2+16m+9})/{2}\right\}$. They also gave a necessary condition and a sufficient condition for a star forest $mK_{1,2} \cup mK_{1,n_1} \cup mK_{1,n_2} \cup \dots \cup mK_{1,n_k}$, $n_1,n_2,\dots,n_k \geq 3$ to be antimagic. 
  In addition, they proved that a star forest with an extra disjoint path is antimagic. 
   Many works related to the antimagicness for product of graphs was discussed by various authors for connected graphs. The readers can refer to the following references: for the Cartesian product \cite{ycheng1,OudoneMCS,WangDM,ChengDM,LiangTCS}, for the lexicographic product \cite{WenhuiAKCE, YingyuMCS, YingyuACTA}, for the corona product \cite{Daykin}, for the join of graphs \cite{Baca,TaoWANG}. However, none of the researchers focused on the antimagicness of the direct product of some graph. 
   

  The direct product of graphs was first considered by  Weichsel \cite{Weichsel} in  1962, which was originally derived from Kronecker product of matrices. There are several names used for the direct product of graphs that are used by different authors. Those are cardinal product, Kronecker product, tensor product, categorical product and graph conjunction. The direct product of graphs $G$ and $H$, denoted by $G \times H$,  is the graph with the vertex sets same as has the Cartesian product of these graphs, i.e.,  $V(G \times H)=V(G) \times V(H)$ such that the vertex pairs $(x,y)$ and $(x',y')$ are adjacent in $G \times H $ if and only if $x$ is adjacent to $x'$ in $G$ and $y$ is adjacent to $y'$ in $H$.
The connectedness of the direct product of two graphs is characterized  in the following theorem.
\begin{theorem} \cite{Weichsel} \label{Connectedwei}
Let $G$ and $H$ be   connected graphs. The direct product   $G \times H$  connected if and only if  either $G$ or $H$ contains an odd cycle.
\end{theorem}
\begin{corollory} \label{Exactly} \cite{Weichsel}
If $G$ and $H$ are connected graphs with no odd cycles then the direct product  $G  \times H$ has exactly two connected components.
\end{corollory}

In this paper we study the antimagicness of the direct product of a star and a path. Our main result is the following.
\begin{theorem}\label{main}
The graph $K_{1,s}\times P_{n}$ is antimagic for all positive integers $s$, $n\ge 2$ except three cases when  $(s,n)\in \{(1,2), (1,3), (2,2)\}$.
\end{theorem}

According to  Corollary  \ref{Exactly} the direct product   $K_{1,n} \times P_m$ has exactly two connected components.
Evidently, when $s=1$   the graph  $K_{1,s}\times P_{n}$ is isomorphic to two copies of the path $P_n$. Trivially, $2P_2$ is not antimagic. In \cite{WANG} Wang, Lui and Li proved that $mP_3$ is not antimagic for $m\ge 2$. Moreover, they proved that the union of two copies of a path on at least four vertices is an antimagic graph. This immediately implies that $K_{1,1}\times P_{n}$ is antimagic if and only if $n\ge 4$.
 
To complete the proof of Theorem  \ref{main} for $s\ge 2$ we distinguish two cases according to the parity on $n$. These cases are discussed in the following two sections.
\section{ A path  on even number of vertices}
First consider a graph $K_{1,s}\times P_{2m+2}$, $s\ge 2$, $m\ge 0$. This graph is disconnected and consists of two isomorphic copies. Let us denote the vertices and edges of $K_{1,s}\times P_{2m+2}$ in the following way.
\begin{align*}
V(K_{1,s}\times P_{2m+2})=&\{a_{i}^j, b_i^{j}, v_i, u_i: i=0, 1,  \dots, m, j=1, 2, \dots, s\},\\
E(K_{1,s}\times P_{2m+2})=&\{a_{i}^jv_i, b_i^{j}u_i : i=0, 1,  \dots, m, j=1, 2, \dots, s\}\\
&\cup \{a_{i}^jv_{i-1}, b_i^{j}u_{i-1} : i= 1, 2 \dots, m, j=1, 2, \dots, s\}.
\end{align*}

Before all else we solve two small cases. More precisely, we consider two special cases when $m=0$, i.e., the direct product $K_{1,s} \times P_{2}$ for $s\ge 2$ and   the case when $s=2$, i.e., the direct product $K_{1,2} \times P_{2m+2}$  is antimagic for $m \geq 1$.  

\begin{lemma}\label{P2}
The graph $K_{1,s} \times P_{2}$  is antimagic for $s \geq 3$.	
\end{lemma}
\begin{proof}
Sampathkumar \cite{Sampathkumar}   proved that if a  connected graph $G$ contains  no odd cycle then $G \times K_2$ is isomorphic to $2G$.
Thus the graph $K_{1,s} \times P_{2}$  is   isomorphic to two copies of the star $K_{1,s}$. Evidently, the graphs $2K_{1,1}$ and $2K_{1,2}\cong 2P_3$ are not antimagic. In \cite{WANG} is proved that $2K_{1,s}$ is antimagic for $s\ge 3$.
\end{proof}

\begin{lemma}\label{even, s=2}
The graph $K_{1,2} \times P_{2m+2}$  is antimagic for $m \geq 1$.	
\end{lemma}
\begin{proof}
Let us define an edge labeling $f$ of $K_{1,2} \times P_{2m+2}$ in the following way:
\begin{align*}
f(a_{0}^{j}v_0) =& 2+j, &&\hspace*{-0cm} \text{for $j=1,2$},\\
f(b_{0}^{j}u_0)=& j, &&\hspace*{-0cm}  \text{for $j=1,2$},\\
f(b_{i}^{j}u_{i-1}) =& 3+2i +4m(j-1), &&\hspace*{-0cm}  \text{for   $i=1, 2, \dots, m$ and $j=1, 2$},\\
f(b_{i}^{j}u_{i}) =& 4+2i +4m(j-1), &&\hspace*{-0cm}  \text{for   $i=1, 2, \dots, m$ and $j=1, 2$},\\
f(a_{i}^{j}v_{i-1}) =& 3+2i +2m(2j-1), &&\hspace*{-0cm}  \text{for   $i=1, 2, \dots, m$ and $j=1, 2$},\\
f(a_{i}^{j}v_{i}) =& 4+2i +2m(2j-1), &&\hspace*{-0cm}  \text{for   $i=1, 2, \dots, m$ and $j=1, 2$}.
\end{align*}
Evidently, the edges are labeled with distinct numbers from $1$ to $8m+4$. Now, we evaluate the induced vertex labels under the function $f.$ The weights of vertices of degree one are
\begin{align*}
wt_{f}(a_0^j)=& f (a_0^j v_0)=2+j,\\
wt_{f}(b_0^j)=& f (u_{0}b_{0}^{j})=j
\end{align*}
for $j=1,2$, thus the weights are  $1, 2, 3$ and $4$.\\
Now we evaluate the weights of vertices of degree 2. For     $i=1, 2, \dots, m$ and $j=1, 2$ we obtain
\begin{align*}
wt_{f} (a_{i}^{j}) =& f(v_{i-1}a_{i}^{j})+f(a_{i}^{j}v_i) = 4m+7+4i+8m(j-1),\\
wt_{f} (b_{i}^{j}) =& f(v_{i-1}b_{i}^{j})+f(b_{i}^{j}v_i) =  7+4i+8m(j-1).
\end{align*}
Thus the weights of vertices of degree 2 are the following odd numbers $11, 15, \dots, 16m+7$. More precisely, all of them are congruent 3 modulo 4.  
The weights of vertices $v_m$ and $u_m$ are divisible by 4, as
\begin{align*}
wt_{f} (v_m) =&   f(a_{m}^{1} v_m) + f(a_{m}^{2} v_m) =12m+8,\\
wt_{f} (u_m) =&   f(b_{m}^{1} u_m)+f(b_{m}^{2} u_m) =   8m+8.
\end{align*}

Now, we evaluate the weights of vertices of degree 4. For $ i = 1,2, \dots, m-1$, we get
\begin{align*}
wt_{f} (v_i) =& \sum_{j=1}^{2} f(a_{i}^{j} v_i) + \sum_{j=1}^{2} f( a_{i+1}^{j}v_i)  =  16m+18+8i,\\
wt_{f} (u_i) =& \sum_{j=1}^{2} f(b_{i}^{j} u_i) + \sum_{j=1}^{2} f( b_{i+1}^{j} u_i)  = 8m+18+8i.
\end{align*}
Which means that they are distinct numbers and all of them are congruent 2 modulo 4.
Finally, 
\begin{align*}
wt_{f} (v_0) =& \sum_{j=1}^{2} f(a_{0}^{j} v_0) + \sum_{j=1}^{2} f(a_{1}^{j} v_0)  = 8m+17,\\
wt_{f} (u_0) =& \sum_{j=1}^{2} f(b_{0}^{j} u_0) + \sum_{j=1}^{2} f(b_{1}^{j} u_0)  = 4m+13.
\end{align*}
Thus they are congruent 1 modulo 4. 
Evidently, all vertex weight induced by the labeling $f$ are distinct numbers thus $f$ is an antimagic labeling of $K_{1,2} \times P_{2m+2}$.
\end{proof}

Now consider the case when $s\ge 3$.
\begin{theorem}\label{even, s>2}   
The graph $K_{1,s}\times P_{2m+2}$ is antimagic for $s\ge 3$, $m\ge 1$.
\end{theorem}
\begin{proof}
Let us define an edge labeling $f_{\varepsilon}$, $\varepsilon\in\{0, 1,\dots, s^2\}$, of $K_{1,s}\times P_{2m+2}$, $s\ge 3$, $m\ge 1$, such that the pendant edges are labeled as follows:
\begin{align}\label{stupen1}
\{f_{\varepsilon}(a_0^jv_0), f_{\varepsilon}(b_0^ju_0): j=1, 2, \dots, s\}=&\{1, 2, \dots, 2s\}
\end{align}
and 
\begin{align}\label{prispevok k v0}
\sum_{j=1}^{s}f_{\varepsilon}(a_0^jv_0)=\frac{s(s+1)}{2}+\varepsilon,
\end{align}
thus 
\begin{align}\label{prispevok k u0}
\sum_{j=1}^{s}f_{\varepsilon}(b_0^ju_0)=s^2+\frac{s(s+1)}{2}-\varepsilon.\end{align}
 The labels of the remaining edges are
\begin{align*}
 f_{\varepsilon}(v_{i-1}a_i^j) =&
        \begin{cases} 
				     2j-1+2is, & \text{for $i\equiv 1\pmod 2$, $1\le i\le m$ and $j=1, 2, \dots, s$},\\
						 2j+2is, & \text{for $i\equiv 0\pmod 2$, $2\le i\le m$ and $j=1, 2, \dots, s$},
\end{cases}\\
 f_{\varepsilon}(a_i^jv_{i}) =&
        \begin{cases} 
				     2j+2is, & \text{for $i\equiv 1\pmod 2$, $1\le i\le m$ and $j=1, 2, \dots, s$},\\
						 2j-1+2is, & \text{for $i\equiv 0\pmod 2$, $2\le i\le m$ and $j=1, 2, \dots, s$},
\end{cases}\\
 f_{\varepsilon}(u_{i-1}b_i^j) =&
        \begin{cases} 
				     2j-1+2is+2ms, & \text{for $i\equiv 1\pmod 2$, $1\le i\le m$ and $j=1, 2, \dots, s$},\\
						 2j+2is+2ms, & \text{for $i\equiv 0\pmod 2$, $2\le i\le m$ and $j=1, 2, \dots, s$},
\end{cases}\\
 f_{\varepsilon}(b_i^ju_{i}) =&
        \begin{cases} 
				     2j+2is+2ms, & \text{for $i\equiv 1\pmod 2$, $1\le i\le m$ and $j=1, 2, \dots, s$},\\
						 2j-1+2is+2ms, & \text{for $i\equiv 0\pmod 2$, $2\le i\le m$ and $j=1, 2, \dots, s$}.
\end{cases}
\end{align*}
Evidently, the edges are labeled with distinct numbers from $1$  to $4ms+2s$.  

Now we evaluate the induced vertex labels under the labeling $f_{\varepsilon}$.
The weights of vertices of degree one are
\begin{align*}
wt_{f_{\varepsilon}}(a_0^j)=& f_{\varepsilon}(a_0^j v_0),\\
wt_{f_{\varepsilon}}(b_0^j)=& f_{\varepsilon}(b_0^ju_0)
\end{align*}
for $j=1, 2, \dots, s$. According to (\ref{stupen1}) we get that they are distinct numbers from $1$ to $2s$.

For the weights of   vertices of $a_i^j$,  $i=1, 2, \dots, m$, $j=1, 2, \dots, s$, we get
\begin{align*}
wt_{f_{\varepsilon}}(a_i^j)=& f_{\varepsilon}(v_{i-1}a_i^j)+f_{\varepsilon}(a_i^jv_{i})= 
				     4j+4is-1
						\end{align*}
thus they are numbers from the set $\{4s+3, 4s+7, \dots, 4ms+4s-1\}$.

The weights of   vertices of $b_i^j$,  $i=1, 2, \dots, m$, $j=1, 2, \dots, s$ are
						\begin{align*}						
						wt_{f_{\varepsilon}}(b_i^j)=& f_{\varepsilon}(u_{i-1}b_i^j)+f_{\varepsilon}(b_i^ju_{i})= 
				    4ms+ 4j+4is-1,  
\end{align*}
i.e., they form the set $\{4ms+4s+3, 4ms+4s+7, \dots, 8ms+4s-1\}$.
Thus the weights of vertices of degree two are all distinct numbers of the form $4s+4k+3$, $k=1, 2, \dots, 2ms$, thus they are all odd, moreover all are congruent 3 modulo 4.  \\
Now we check the weights of vertices of degree $2s$. We get
\begin{align*}
wt_{f_{\varepsilon}}(v_m)=& \sum_{j=1}^{s} f_{\varepsilon}(a_m^jv_m)=
     \begin{cases}
		            2ms^2+s^2+s , & \text{when $m$ is  odd},\\
								2ms^2+s^2 , & \text{when $m$ is  even},
			\end{cases}					\\
wt_{f_{\varepsilon}}(u_m)=& \sum_{j=1}^{s} f_{\varepsilon}(b_m^ju_m)=
     \begin{cases}
		            4ms^2+s^2+s , & \text{when $m$ is  odd},\\
								4ms^2+s^2 , & \text{when $m$ is  even}.
			\end{cases}					
\end{align*}
Now we evaluate the weights of vertices of degree $4s$. For $v_i$, $u_i$, $i=1, 2, \dots, m-1$ we get
\begin{align*}
wt_{f_{\varepsilon}}(v_i)=& \sum_{j=1}^{s} f_{\varepsilon}(a_i^jv_i)+\sum_{j=1}^{s} f_{\varepsilon}(v_ia_{i+1}^j)=
     \begin{cases}
		            4s^2+4is^2+2s ,  & \text{when $i$ is  odd, $1\le i\le m-1$},\\
								4s^2+4is^2 ,  & \text{when $i$ is  even, $2\le i\le m-1$},
			\end{cases}					\\
wt_{f_{\varepsilon}}(u_i)=& \sum_{j=1}^{s} f_{\varepsilon}(b_i^ju_i)+\sum_{j=1}^{s} f_{\varepsilon}(u_ib_{i+1}^j)=
     \begin{cases}
		          4ms^2+  4s^2+4is^2+2s , \\ \hspace*{20mm} \text{when $i$ is  odd, $1\le i\le m-1$},\\
							4ms^2+	4s^2+4is^2 , \\ \hspace*{20mm} \text{when $i$ is  even, $2\le i\le m-1$},
			\end{cases}		
\end{align*}
thus all these weights are even numbers. Finally, according to (\ref{prispevok k v0}) and (\ref{prispevok k u0}) we get
\begin{align*}
wt_{f_{\varepsilon}}(v_0)=& \sum_{j=1}^{s} f_{\varepsilon}(a_0^jv_0)+\sum_{j=1}^{s} f_{\varepsilon}(v_0a_{1}^j)=3s^2+\frac{s(s+1)}2+\varepsilon,\\
wt_{f_{\varepsilon}}(u_0)=& \sum_{j=1}^{s} f_{\varepsilon}(b_0^ju_0)+\sum_{j=1}^{s} f_{\varepsilon}(u_0b_{1}^j)=4s^2+\frac{s(s+1)}2+2ms^2-\varepsilon.
\end{align*}

Note that only the weights of vertices $u_0$, $v_0$  and $a_0^j$, $b_0^j$, $j=1, 2, \dots, s$,  depend on the value of $\varepsilon$.

Evidently, the weights of vertices of degree 1 are distinct and they are different (smaller) from all the other vertex weights.

Moreover, as the weights of the vertices of degree $2$ are odd, they are different from the weights of the vertices $v_i$, $u_i$, $i=1, 2, \dots, m-1$, as they are all even numbers. 

It is easy to see that
\begin{align*}
 wt_{f_{\varepsilon}}(v_0)&<  wt_{f_{\varepsilon}}(v_1)< wt_{f_{\varepsilon}}(v_2)< \dots< wt_{f_{\varepsilon}}(v_{m-1}) < wt_{f_{\varepsilon}}(u_1) < wt_{f_{\varepsilon}}(u_2) <\dots   < wt_{f_{\varepsilon}}(u_{m-1}),
\\
  wt_{f_{\varepsilon}}(v_m) &<  wt_{f_{\varepsilon}}(u_0) < wt_{f_{\varepsilon}}(u_1),\\
  wt_{f_{\varepsilon}}(v_0)&<   wt_{f_{\varepsilon}}(v_m) < wt_{f_{\varepsilon}}(u_m).
\end{align*}
Now we prove that also the other vertex weights are distinct. To prove it we have to show the following.
\begin{itemize}
 
\item  $wt_{f_{\varepsilon}}(u_m)\ne wt_{f_{\varepsilon}}(u_i)$ for every $i=1, 2, \dots, m-1$.\\

This follows from the fact that $$wt_{f_{\varepsilon}}(u_m)\le 4ms^2+s^2+s< 4ms^2+8s^2+2s=wt_{f_{\varepsilon}}(u_1)<wt_{f_{\varepsilon}}(u_2)<\dots <wt_{f_{\varepsilon}}(u_{m-1}).$$

\item  $wt_{f_{\varepsilon}}(u_m)\ne wt_{f_{\varepsilon}}(v_i)$ for every $i=1, 2, \dots, m-1$.\\
By contradiction. Consider that for some $t\in\{1, 2, \dots, m-1\}$ holds the equality  $wt_{f_{\varepsilon}}(u_m)= wt_{f_{\varepsilon}}(v_t)$. 
  We distinguish two subcases.

When $m$ is odd then 
\begin{align*}
4ms^2+s^2+s=wt_{f_{\varepsilon}}(u_m)=& wt_{f_{\varepsilon}}(v_t)=\begin{cases}
																																	  4s^2+4ts^2+2s , & \text{when $t$ is  odd},\\
																																		4s^2+4ts^2 , & \text{when $t$ is  even},
																																	\end{cases}\\
s(4m-4t-3)=&			\begin{cases}
																																	 1 , & \text{when $t$ is  odd},\\
																																	-1, & \text{when $t$ is  even}.
																																	\end{cases}																														
\end{align*}
However, this is not possible when $s\ge 2$.

When $m$ is even then 
\begin{align*}
4ms^2+s^2=wt_{f_{\varepsilon}}(u_m)=& wt_{f_{\varepsilon}}(v_t)=\begin{cases}
																																	  4s^2+4ts^2+2s , & \text{when $t$ is  odd},\\
																																		4s^2+4ts^2 , & \text{when $t$ is  even},
																																	\end{cases}\\
s(4m-4t-3)=&			\begin{cases}
																																	 2 , & \text{when $t$ is  odd},\\
																																	0, & \text{when $t$ is  even}.
																																	\end{cases}																														
\end{align*}
This is not possible when $s\ge 3$.

\item  $wt_{f_{\varepsilon}}(u_m)$ is distinct from the weights of vertices of degree 2.

When $m$ is odd then $wt_{f_{\varepsilon}}(u_m)=4ms^2+s(s+1)$ is even and thus it is different from the weights of  vertices of degree 2 as these weights are odd.

When $m$  and $s$ are both  even then $wt_{f_{\varepsilon}}(u_m)=4ms^2+s^2$ is even. Thus it is  different from the weights of  vertices of degree 2.

When $m$ is even and $s$ is odd then $wt_{f_{\varepsilon}}(u_m)=4ms^2+s^2\equiv 1\pmod 4$. As the weights of the vertices of degree 2 are congruent 3 modulo 4 we have that are distinct.

\item  $wt_{f_{\varepsilon}}(v_m)\ne wt_{f_{\varepsilon}}(v_i)$ for every $i=1, 2, \dots, m-1$.
 
By contradiction. Consider that for some $t\in\{1, 2, \dots, m-1\}$ holds the equality  $wt_{f_{\varepsilon}}(v_m)= wt_{f_{\varepsilon}}(v_t)$. 
  We distinguish two subcases.

When $m$ is odd then 
\begin{align*}
2ms^2+s^2+s=wt_{f_{\varepsilon}}(v_m)=& wt_{f_{\varepsilon}}(v_t)=\begin{cases}
																																	  4s^2+4ts^2+2s , & \text{when $t$ is  odd},\\
																																		4s^2+4ts^2 , & \text{when $t$ is  even},
																																	\end{cases}\\
s(2m-4t-3)=&			\begin{cases}
																																	 1 , & \text{when $t$ is  odd},\\
																																	-1, & \text{when $t$ is  even}.
																																	\end{cases}																														
\end{align*}
However, this is not possible when $s\ge 2$.

When $m$ is even then 
\begin{align*}
2ms^2+s^2=wt_{f_{\varepsilon}}(v_m)=& wt_{f_{\varepsilon}}(v_t)=\begin{cases}
																																	  4s^2+4ts^2+2s , & \text{when $t$ is  odd},\\
																																		4s^2+4ts^2 , & \text{when $t$ is  even},
																																	\end{cases}\\
s(2m-4t-3)=&			\begin{cases}
																																	 2 , & \text{when $t$ is  odd},\\
																																	0, & \text{when $t$ is  even}.
																																	\end{cases}																														
\end{align*}
This is not possible when $s\ge 3$.

\item  $wt_{f_{\varepsilon}}(v_m)$ is distinct from the weights of vertices of degree 2.

When $m$ is odd then $wt_{f_{\varepsilon}}(v_m)=2ms^2+s(s+1)$ is even and thus it is different from the weights of  vertices of degree 2 as these weights are odd.

When $m$  and $s$ are both  even then $wt_{f_{\varepsilon}}(v_m)=2ms^2+s^2$ is even. Thus it is  different from the weights of  vertices of degree 2.

When $m$ is even and $s$ is odd then $wt_{f_{\varepsilon}}(v_m)=2ms^2+s^2\equiv 1\pmod 4$. As the weights of vertices are congruent 3 modulo 4 we have that they are distinct.

\item
Now we prove that for at least one integer $\varepsilon^*$  from the set $ \{0, 1, 2, 3\}$ under the labeling $f_{\varepsilon^*}$ the weight of the vertex $v_0$ is different from the 
 weights of vertices of degree 2, the weight of the vertex $u_0$ is different from the 
 weights of vertices of degree 2 and also $wt_{f_{\varepsilon^*}}(u_0)\ne wt_{f_{\varepsilon^*}}(v_i)$ for every $i=1, 2, \dots, m-1$.

But this follows from the fact that the 
difference 
between two  weights of vertices of degree 2 is   four and the difference between the weights of vertices $v_i$,  $i=1, 2, \dots, m-1$, is at least $4s^2-2s$. 
\end{itemize}

This concludes the proof.
\end{proof}

Figure \ref{Figure 1} illustrates an  antimagic labeling of $K_{1,3}\times P_{8}$.
\begin{figure}[ht!]
\begin{center}
\includegraphics{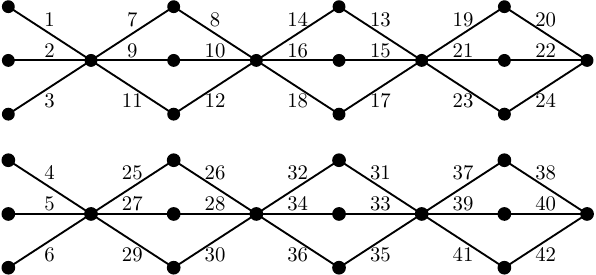} 
\caption{An antimagic labeling of  of $K_{1,3}\times P_{8}$.}
\label{Figure 1}
\end{center}
\end{figure}

Combing Lemmas \ref{P2}, \ref{even, s=2} and Theorem \ref{even, s>2}   we get the following result for the direct product of a star with a path on  even number of vertices.

\begin{theorem}\label{even}   
The graph $K_{1,s}\times P_{2m+2}$ is antimagic for $s\ge 2$, $m\ge 0$ except the case when $(s,n)=(2,2)$.
\end{theorem}

\section{ A path  on odd number of vertices}
In this section we consider the direct product of a star with a path on odd number of vertices. Also in this case the graph $K_{1,s} \times P_{2m+1}$ consist of two connected components however, these components are not isomorphic.
Let us denote the vertices and edges of $K_{1,s} \times P_{2m+1}$, $s\ge 2$, $m\ge 1$,  in the following way
\begin{align*}
V(K_{1,s} \times P_{2m+1}) = &	\left\{a_{i}^{j},u_i : i=0,1,\dots,m, j = 1,2, \dots, s\right\}\\&\cup \left\{b_{i}^{j}  : i=1,2,\dots,m, j = 1,2, \dots, s\right\} \cup \left\{v_{i}: i=0,1,\dots, m-1\right\},\\
E(K_{1,s} \times P_{2m+1}) =& \left\{a_{i}^{j}v_i, a_{i+1}^{j}v_{i}: i=0,1,\dots,m-1, j=1,2,\dots,s\right\}
 \\ &\cup \left\{b_{i}^{j}u_i,b_{i}^{j}u_{i-1} : i=1,2, \dots, m, j = 1,2, \dots, s\right\}. 
\end{align*}

First we solve the small cases $K_{1,s} \times P_{3}$ and $K_{1,s} \times P_{5}$.
\begin{lemma}\label{odd, P3}
The graph $K_{1,s} \times P_{3}$ is antimagic for $s\geq 2$.
\end{lemma}
\begin{proof}
Let us define an edge labeling $f$ of $K_{1,s} \times P_{3}$ such that:
\begin{align*}
f(a_{0}^{j} v_0)=&j, &&\hspace*{-0cm}\text{for $j=1, 2, \dots, s$},\\
f(v_0 a_{1}^{j})=&s+j, &&\hspace*{-0cm}\text{for $j=1, 2, \dots, s$},\\
f(u_1 b_{1}^{j})=& 2s+2j-1, &&\hspace*{-0cm}\text{for $j=1, 2, \dots, s$},\\
f(u_0 b_{1}^{j})=& 2s+2j, &&\hspace*{-0cm}\text{for $j=1, 2, \dots, s$}.
\end{align*}
The induced vertex weights are
\begin{align*}
wt_{f} (a_0^j)=& f (a_0^j v_0)=j,&&\hspace*{-0cm}\text{for $j=1, 2, \dots, s$},\\
wt_{f} (a_1^j)=& f(v_{0}a_{1}^{j})=s+j,&&\hspace*{-0cm}\text{for $j=1, 2, \dots, s$},\\
wt_{f}(b_{1}^{j})=& \sum_{j=1}^{s} f (u_{0} b_{1}^{j}) +  \sum_{j=1}^{s} f (b_{1}^{j} u_{1}) = 4s+4j-1,  &&\hspace*{-0cm}\text{for $j=1, 2, \dots, s$},\\
wt_{f}(v_{0})=& \sum_{j=1}^{s} f (a_{0}^{j} v_{0}) +  \sum_{j=1}^{s} f (a_{1}^{j} v_{0}) = 2s^2+s,\\
wt_{f}(u_{1}) =& \sum_{j=1}^{s} f (u_{1} b_{1}^{j} ) = 3s^{2},\\
wt_{f}(u_{0}) =& \sum_{j=1}^{s} f (u_{0} b_{1}^{j} ) = 3s^{2}+s.
\end{align*}
Clearly, for $s\geq 2$ all the  weights are distinct.
\end{proof}

\begin{lemma}\label{odd, P5}
The graph $K_{1,s} \times P_{5}$ is antimagic for $s\geq 2$.
\end{lemma}
\begin{proof}
In this case consider an edge labeling $f$ of $K_{1,s} \times P_{5}$ in the following way
\begin{align*}
f(a_{0}^{j} v_0)=&j, &&\hspace*{-0cm}\text{for $j=1, 2, \dots, s$},\\
f(v_1 a_{2}^{j})=&s+j, &&\hspace*{-0cm}\text{for $j=1, 2, \dots, s$},\\
f(v_1 a_{1}^{j})=& 2s+2j-1, &&\hspace*{-0cm}\text{for $j=1, 2, \dots, s$},\\
f(v_0 a_{1}^{j})=& 2s+2j, &&\hspace*{-0cm}\text{for $j=1, 2, \dots, s$},\\
f(u_1 b_{1}^{j}) =& 4s+2j -1, &&\hspace*{-0cm} \text{for $j=1,2, \dots, s$}, \\
f(u_0 b_{1}^{j}) =& 4s+2j , &&\hspace*{-0cm} \text{for $j=1,2, \dots, s$}, \\
f(u_1 b_{2}^{j}) =& 6s+2j-1, &&\hspace*{-0cm} \text{for $j=1,2, \dots, s$}, \\
f(u_2 b_{2}^{j}) =& 6s+2j, &&\hspace*{-0cm} \text{for $j=1,2, \dots, s$}.
\end{align*}
Evidently $f$ is a bijection. The vertex weights are
\begin{align*}
wt_{f} (a_0^j)=& f (a_0^j v_0)=j,&&\hspace*{-0cm}\text{for $j=1, 2, \dots, s$},\\
wt_{f}(a_{2}^{j})=& f (a_{2}^{j}v_1)=s+j,&&\hspace*{-0cm}\text{for $j=1, 2, \dots, s$},\\
wt_{f}(a_{1}^{j}) =& \sum_{j=1}^{s} f (v_{0} a_{1}^{j} ) + \sum_{j=1}^{s} f (v_{1} a_{1}^{j} ) = 4s+4j-1,&&\hspace*{-0cm}\text{for $j=1, 2, \dots, s$},\\
wt_{f}(b_{1}^{j}) =& \sum_{j=1}^{s} f (u_{0} b_{1}^{j} ) + \sum_{j=1}^{s} f (u_{1} b_{1}^{j} ) = 8s+4j-1,&&\hspace*{-0cm}\text{for $j=1, 2, \dots, s$},\\
wt_{f}(b_{2}^{j}) =& \sum_{j=1}^{s} f (u_{1} b_{2}^{j} ) + \sum_{j=1}^{s} f (u_{2} b_{2}^{j} ) = 12s+4j-1,&&\hspace*{-0cm}\text{for $j=1, 2, \dots, s$},\\
wt_{f}(v_{0}) =& \sum_{j=1}^{s} f (v_{0} a_{0}^{j} ) + \sum_{j=1}^{s} f (v_{0} a_{1}^{j} ) = \frac{7s^2 + 3s}{2} , \\
wt_{f}(v_{1}) =& \sum_{j=1}^{s} f (v_{1} a_{1}^{j} ) + \sum_{j=1}^{s} f (v_{1} a_{2}^{j} ) = \frac{9s^2 + s}{2} ,\\ 
wt_{f}(u_0) =& \sum_{j=1}^{s} f (u_{0} b_{1}^{j} ) = 5s^2 +s, \\
wt_{f}(u_2) =& \sum_{j=1}^{s} f (u_{2} b_{2}^{j} ) = 7s^2 +s,\\ 
wt_{f}(u_1) =& \sum_{j=1}^{s} f (u_{1} b_{1}^{j} ) + \sum_{j=1}^{s} f (u_{1} b_{2}^{j} )  = 12s^2. 
\end{align*}
It is clear from the above vertex sums that for $s\ge 2$ the weights of the vertices are distinct.
\end{proof}

\begin{theorem}\label{odd}
The graph $K_{1,s} \times P_{2m+1}$ is antimagic for $s\geq 2$, $m \geq 3$.
\end{theorem}

\begin{proof}
Let us define an edge labeling $f_{\varepsilon}$, ${\varepsilon} \in \left\{0,1 \dots, s^2\right\}$, of $K_{1,s} \times P_{2m+1}$, $s\ge 2$, $m\ge 3$, such that the pendant edges are labeled as follows:
\begin{align}\label{odd1}
\{f_{\varepsilon}(a_{0}^{j} v_0), f_{\varepsilon}(v_{m-1}a_{m}^{j}): j=1, 2, \dots, s\}=&\{1, 2, \dots, 2s\}
\end{align}
and
\begin{align}\label{odd2}
\sum_{j=1}^{s}f_{\varepsilon}(a_0^jv_0)=\frac{s(s+1)}{2}+\varepsilon,
\end{align}
thus 
\begin{align}\label{odd3}
\sum_{j=1}^{s}f_{\varepsilon}(a_{m}^{j}v_{m-1})=s^2+\frac{s(s+1)}{2}-\varepsilon.
\end{align}
The labels of the remaining edges are 
\begin{align*}
 f_{\varepsilon}(v_{i-1}a_{i}^{j}) =&
        \begin{cases} 
				   2j+  2is, & \text{for $i\equiv 1\pmod 2$, $1\le i\le m-1$ and $j=1, 2, \dots, s$},\\
					2j +2is-1, & \text{for $i\equiv 0\pmod 2$, $2\le i\le m-1$ and $j=1, 2, \dots, s$},
\end{cases}\\
f_{\varepsilon}(a_i^j v_{i}) =&
        \begin{cases} 
				     2j +2is-1, & \text{for $i\equiv 1\pmod 2$, $1\le i\le m-1$ and $j=1, 2, \dots, s$},\\
						 2j+  2is, & \text{for $i\equiv 0\pmod 2$, $2\le i\le m-1$ and $j=1, 2, \dots, s$},
\end{cases}\\
f_{\varepsilon}(u_{i-1}b_i^j)=&
        \begin{cases} 
				  2j+2is+2(m-1)s, & \text{for $i\equiv 1\pmod 2$, $1\le i\le m$ and $j=1, 2, \dots, s$},\\
					2j+2is+2(m-1)s-1, & \text{for $i\equiv 0\pmod 2$, $2\le i\le m$ and $j=1, 2, \dots, s$},
\end{cases}\\
f_{\varepsilon}(b_i^ju_{i})=&
        \begin{cases} 
				     2j+2is+2(m-1)s-1, & \text{for $i\equiv 1\pmod 2$, $1\le i\le m$ and $j=1, 2, \dots, s$},\\
						 2j+2is+2(m-1)s & \text{for $i\equiv 0\pmod 2$, $2\le i\le m$ and $j=1, 2, \dots, s$}.
\end{cases}
\end{align*}
Evidently, the edge labels   are  distinct numbers from $1$ to $4ms$. Now, we evaluate the induced vertex labels under the function $f_\varepsilon$. The weights of vertices of degree one is
\begin{align*}
wt_{f_{\varepsilon}}(a_0^j)= &f_{\varepsilon}(a_0^j v_0),\\
wt_{f_{\varepsilon}}(a_0^m)=& f_{\varepsilon}(v_{m-1}a_{0}^{j})
\end{align*}
for $j=1,2, \dots, s$. According to (\ref{odd1}) we get that they are distinct numbers from $1$ to $2s$.

For the weights of vertices of $a_{i}^{j}$, $i= 1,2, \dots, m-1$, $j = 1,2, \dots, s$, we obtain
$$
wt_{f_{\varepsilon}} (a_{i}^{j}) = f_{\varepsilon} (v_{i-1}a_{i}^{j}) + f_{\varepsilon} (a_{i}^{j} v_i) = 4j+4is-1
$$
thus they are odd numbers from the set $\left\{4s+3,4s+7, \dots, 4ms-1\right\}$, i.e., they are congruent 3 modulo 4.

For the weights of vertices of $b_{i}^{j}$, $i = 1,2, \dots, m$, $j = 1,2, \dots, s$, we have
$$
wt_{f_{\varepsilon}} (b_{i}^{j}) = f_{\varepsilon} (u_{i-1} b_{i}^{j}) + f_{\varepsilon} (b_{i}^{j} u_i) = 4ms+4is+4j-4s-1
$$
which are odd numbers from the set $\left\{4ms+3, 4ms+7, \dots, 8ms-1\right\}$. Again these numbers are   congruent 3 modulo 4.

Now we check the weights of vertices of degree $4s$. For $v_i$, $i = 1,2, \dots, m-2,$ and $u_i$, $i=1,2, \dots, m-1$, we get
\begin{align*}
wt_{f_{\varepsilon}}(v_i)=& \sum_{j=1}^{s} f_{\varepsilon}(a_i^jv_i) + \sum_{j=1}^{s} f_{\varepsilon}(v_{i} a_{i+1}^{j}) =
     \begin{cases}
		            4is^2 +4s^2  , & \text{when $i$ is  odd, $1 \leq i \leq m-2$},\\
		            4is^2+4s^2+2s , & \text{when $i$ is  even, $2 \leq i \leq m-2$},
			\end{cases}					\\
wt_{f_{\varepsilon}}(u_i)=& \sum_{j=1}^{s} f_{\varepsilon}(b_i^j u_i) + \sum_{j=1}^{s} f_{\varepsilon}( u_{i}b_{i+1}^j) =
     \begin{cases}
		            4ms^2+4s^2 i , & \text{when $i$ is  odd}, 1 \leq i \leq m-1\\
								4ms^2 + 4s^2 i +2s , & \text{when $i$ is  even}, 2 \leq i \leq m-1.
			\end{cases}					
\end{align*}
Next we evaluate the weights of vertices of degree $s$. We get
\begin{align*}
wt_{f_{\varepsilon}}(u_0)=& \sum_{j=1}^{s} f_{\varepsilon} (b_{1}^{j} u_0) = 2ms^2 + s^2 +s, \\
wt_{f_{\varepsilon}}(u_m)=& \sum_{j=1}^{s} f_{\varepsilon}(b_m^j u_m)  =
     \begin{cases}
		            4ms^2 - s^2 , & \text{when $m$ is  odd},\\
								4ms^2 - s^2 + s , & \text{when $m$ is  even}.
			\end{cases}					
\end{align*}
Finally, according to (\ref{odd2}) and (\ref{odd3}). We get
\begin{align*}
wt_{f_{\varepsilon}}(v_0)=& \sum_{j=1}^{s} f_{\varepsilon} (a_{0}^{j} v_0) + \sum_{j=1}^{s} f_{\varepsilon} (a_{1}^{j} v_0) = \frac{s(s+1)}{2} +3s^2 +s + \varepsilon,\\
wt_{f_{\varepsilon}}(v_{m-1})=& \sum_{j=1}^{s} f_{\varepsilon}(a_m^j v_{m-1}) +  \sum_{j=1}^{s} f_{\varepsilon}(a_{m-1}^j v_{m-1}) =
     \begin{cases}
		            2ms^2 + \frac{s(s+1)}{2} - \varepsilon , & \text{when $m$ is  even},\\
								2ms^2 + \frac{s(s+1)}{2} +s - \varepsilon , & \text{when $m$ is  odd}.
			\end{cases}					
\end{align*}
Note that only the weights of vertices $v_0$, $v_{m-1}$ and $a_{0}^{j}$, $a_{m}^{j}$, $j = 1,2, \dots, s$, depend on the value of $\varepsilon$. 

Evidently, the weights of the vertices of degree $1$ are distinct and they are different (smaller) from all the other vertex weights.

Moreover, as the weights of the vertices of degree $2$ are odd, they are different from the weights of the vertices $v_i$, $i = 1,2, \dots, m-2$ and $u_i$, $i = 1,2, \dots, m-1$, as they are all even numbers. Moreover, it is easy to see that
\begin{align*}
wt_{f_\varepsilon} (v_0) <& wt_{f_\varepsilon} (v_1) < wt_{f_\varepsilon} (v_2) < \dots < wt_{f_\varepsilon} (v_{m-2}) < wt_{f_\varepsilon} (u_1) < wt_{f_\varepsilon} (u_2) < \dots < wt_{f_\varepsilon} (u_{m-1}),\\
wt_{f_\varepsilon} (u_0) <& wt_{f_\varepsilon} (u_m) < wt_{f_\varepsilon} (u_1).
\end{align*}
Moreover, for $m\ge 3$ and $s\ge 2$ we get that for every ${\varepsilon} \in \left\{0,1 \dots, s^2\right\}$ also holds
$$
wt_{f_\varepsilon} (v_0) < wt_{f_\varepsilon} (v_{m-1}) < wt_{f_\varepsilon} (u_0).
$$
Now, we prove that  also the other vertex weights are distinct. To prove it we show the following. 
\begin{itemize}

\item $wt_{f_\varepsilon} (u_m)  \neq wt_{f_\varepsilon} (v_i)$ for every $i = 1,2, \dots, m-2.$ This follows from the fact that,
$$
 wt_{f_\varepsilon} (v_{m-2}) \le   4ms^2  -4s^2 +2s < 4ms^2 - s^2 \le  wt_{f_\varepsilon} (u_m).
$$ 
\item 
 $wt_{f_\varepsilon} (u_m)$ is distinct from the weight of the vertices of degree $2$.
When $m$ is even then $wt_{f_\varepsilon} (u_m) = 4ms^2 -s^2 +s$ is even.
Also when $m$ is odd and $s$ is even we get that $wt_{f_\varepsilon} (u_m) = 4ms^2 -s^2$ is even. Thus in these cases $wt_{f_\varepsilon} (u_m)$  is different from the weights of vertices of degree $2$  as these weights are odd.

When both $m$ and $s$ are odd, $s\ge 3$ then 
\begin{align*}
wt_{f_\varepsilon} (u_m) = s^2(4m-1)\ge 3s(4m-1)>8ms>8ms-1=wt_{f_\varepsilon} (b_m^s).
\end{align*}
  Thus also in this case  $wt_{f_\varepsilon} (u_m)$ is different from  the weights of the vertices of degree 2.

\item  $wt_{f_\varepsilon} (u_0)  \neq wt_{f_\varepsilon} (v_i)$ for every $i = 1,2, \dots , m-2$.

By contradiction. Consider that there exists  $t \in \left\{1,2, \dots, m-2\right\}$ such that $wt_{f_\varepsilon} (u_0) = wt_{f_\varepsilon} (v_t)$. Then
\begin{align*}
2ms^2 + s^2 +s =&  wt_{f_\varepsilon} (u_0) = wt_{f_\varepsilon} (v_t) = 
     \begin{cases}
		           4ts^2 + 4s^2 +2s , & \text{when $t$ is  even},\\
								4ts^2 + 4s^2 , & \text{when $t$ is  odd},
			\end{cases}\\					
s(2m-4t-3) =& 
     \begin{cases}
		           1 , & \text{when $t$ is  even},\\
								-1 , & \text{when $t$ is  odd}.
			\end{cases}					
\end{align*}
However, this is not possible when $s \geq 2$.

\item 
 $wt_{f_\varepsilon} (u_0)$ is distinct from the weight of the vertices of degree $2$.
 
This follows from the fact that $wt_{f_\varepsilon} (u_0) = 2ms^2 + s^2 + s$ is even and the weights of the vertices of degree $2$   are odd.

\item 
Now, we prove that for at least one integer $\varepsilon^{*}$ from the set $\left\{0,1,2,3\right\}$ under the labeling $f_{\varepsilon^*}$ 
   the weight of the vertex $v_0$ is different from the weights of the vertices of degree $2$,
   the weight of the vertex $v_{m-1}$ is different from the weight of the vertices of degree $2$ 
   and also $wt_{f_\varepsilon} (v_{m-1})\ne wt_{f_\varepsilon} (v_i)$ for every $i = 1,2, \dots, m-2$.

But this follows from the fact that the difference between two weights of vertices of degree $2$ is  four and the difference between the weight of vertices $v_i$, $i = 1,2, \dots, m-2, $ is at least $4s^2+2s$.
\end{itemize}

Figure \ref{Figure 2} illustrates an  antimagic labeling of $K_{1,3}\times P_{7}$.
\begin{figure}[ht!]
\begin{center}
\includegraphics{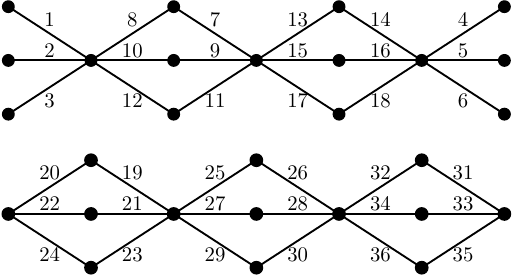} 
\caption{An antimagic labeling of  of $K_{1,3}\times P_{7}$.}
\label{Figure 2}
\end{center}
\end{figure}

Combining the previous the proof is completed.
\end{proof}
According to Lemmas \ref{odd, P3}, \ref{odd, P5} and Theorem \ref{odd} we obtain the following result for the direct product of a star and an odd path.
\begin{theorem}\label{odd all}
The graph $K_{1,s} \times P_{2m+1}$ is antimagic for $s\geq 2$, $m \geq 1$.
\end{theorem}

\section{ Concluding remarks}
Our main motivation in this paper is to study antimagicness of the disconnected graphs which are constructed using some known graph operations. We used the direct product as an operation to construct new classes of disconnected graphs. We gave a characterization of antimagicness of the  direct product of a star  and a path.
As the main result we obtained that the graph   $K_{1,s}\times P_{n}$ is antimagic for all positive integers $s$, $n\ge 2$ except three cases when  $(s,n)\in \{(1,2), (1,3), (2,2)\}$.
Finally, we conclude our paper with the following question. \\
\textbf{Problem 1.}  What are the other classes of disconnected graphs constructed from the direct product of graphs or from any graph operation  which admit an antimagic labeling?

\section{  Acknowledgement}
The research for this article was supported by the Slovak Research and Development Agency under the contract APVV-19-0153 and by VEGA 1/0243/23.

\end{document}